\documentclass[a4paper,11pt]{article}
\usepackage{graphicx}
\usepackage{amsfonts,amsmath,amsthm}
\usepackage{hyperref}

\newtheorem{theorem}{Theorem}
\newtheorem{lemma}[theorem]{Lemma}

\newtheorem{proposition}[theorem]{Proposition}
\theoremstyle{definition}

\newtheorem{remark}[theorem]{Remark}
\newtheorem{claim}[theorem]{Claim}

\newcommand{\Field}{\mathcal{F}}

\newcommand{\ppp}{\pi}
\newcommand{\one}{\boldsymbol{1}}
\newcommand{\var}{\mathbb{V}\mathrm{ar}}
\newcommand{\estgrad}[1]{\widehat{\nabla_{#1}P}}

\newcommand{\gradif}[1]{\widetilde{\nabla_{#1}^{c}p}}
\newcommand{\tapp}{\widehat{\theta}}
\newcommand{\Uad}{U^{\mathrm{ad}}}
\newcommand{\EE}{\mathbb{E}}
\newcommand{\II}{\mathbb{I}_{\RR^{+}}}

\newcommand{\PP}{\mathbb{P}}
\newcommand{\UU}{\mathcal{U}}

\newcommand{\RR}{\mathbb{R}}
\newcommand{\NN}{\mathbb{N}}
\newcommand{\scal}[2]{\left\langle#1,#2\right\rangle}

\title{Stochastic Programming with Probability Constraints}
\author{Laetitia Andrieu\thanks{EDF R\&{}D, D\'{e}p.~OSIRIS,1 avenue du G\'{e}n\'{e}ral de Gaulle, 92141 Clamart Cedex, France}, Guy Cohen\thanks{CERMICS-ENPC, 6--8 avenue Blaise Pascal, 77455 Marne la VallŽe Cedex 2, France, \texttt{guy.cohen@mail.enpc.fr}},  Felisa J. V\'{a}zquez-Abad\thanks{Dept.~Math.~\&~Stat., University of Melbourne, 3010 Victoria, Australia}}
\date{July 31, 2007}
\begin{document}

\maketitle
\begin{abstract}
In this work we study optimization problems subject to a failure constraint. This constraint is expressed in terms of a condition that causes failure, representing a physical or technical breakdown. We formulate the problem in terms of a probability constraint, where the level of ``confidence'' is a modelling parameter and has the interpretation that the probability of failure should not exceed that level.  Application of the stochastic Arrow-Hurwicz algorithm poses two difficulties: one is structural and arises from the lack of convexity of the probability constraint, and the other is the estimation of the gradient of the probability constraint. We develop two gradient estimators with decreasing bias via a convolution method and a finite difference technique, respectively, and we provide a full analysis of convergence of the algorithms. Convergence results are used to tune the parameters of the numerical algorithms in order to achieve best convergence rates, and numerical results are included via an example of application in finance. 
\end{abstract}

\begin{description}
  \item[Keywords.] 
Probability constraints, stochastic programming, stochastic gradient algorithm, stochastic approximation
\end{description}

\thispagestyle{empty}
\section{Introduction}
\subsection{Constrained Optimization in a Stochastic Setting}
Optimization Theory provides  a convenient approach to formulate and solve problems involving conflicting objectives, which is generally the challenge present in decision making situations.  The main idea is to aggregate  as many objectives as possible into a single objective function, which may be straightforward when those objectives are amenable to an expression into a common unit, say, a currency unit as dollar or euro. In this objective aggregation, weights are allocated to each term in order to reflect preferences or priorities. However, there might be other objectives that can hardly be expressed in a unit commensurable with the previous ones (examples to come hereafter). In such a case, it is better to introduce those other objectives through constraints, that is, each such objective should not exceed a prescribed level. The constraint levels are set a  priori, as are the weights for the different terms in the cost function. 

Duality Theory provides the tools to evaluate the sensitivity of the optimal solution (cost) to those prescribed constraint levels. In mathematical terms, let $u$ be the decision variable in a Hilbert space \(\UU\), $J:\UU\to\RR$ the cost function, and \(\Theta:\UU\to\RR^{d}\) the constraint function.  We consider problems of the type:
\begin{equation}
\label{eq-pb}
\min_{u\in\Uad} J(u)\quad \text{s.t.}\quad \Theta(u)\leq \alpha\;,
\end{equation}
where \(\Uad\) is an ``admissible'' or ``feasible'' closed convex subset of \(\UU\) and inequalities in the constraints involving \(\Theta\) are understood componentwise.
Introduce the multiplier \(\lambda\) (in \(\RR^{d}_{+}\))
and the Lagrangian
\begin{equation}
\label{eq-lag}
L(u,\lambda)= J(u)+\scal{\lambda}{\Theta(u)-\alpha}\;,
\end{equation}
where \(\scal{\cdot}{\cdot}\) denotes the scalar product. Kuhn-Tucker optimality conditions characterize an optimal multiplier  \(\lambda^{\sharp}\) which can be interpreted as the sensitivity of the optimal cost function \(J(u^{\sharp})\) (where \(u^{\sharp}\) denotes the solution of problem~\eqref{eq-pb}) with respect to \(\alpha\) (up to a change of sign).

When random factors  affect the outcomes of a decision, a
classical approach is to assume that the probability distribution of those
factors is known and to appeal
to \emph{stochastic} optimization. Call \(\xi\) the corresponding random variable, then the objective function is usually expressed in terms of an expectation of some cost function
of the form \(J(u)=\EE \big( j(u,\xi) \big )\). 

In the stochastic situation, modelling choices for aggregation of objectives, weights and constraints are similar to the deterministic case. However a new question also arises regarding constraints, namely, constraints can be formulated in various ways: ``almost surely'', ``in expectation'', ``in probability'', etc.

The first possibility (``almost sure'' constraints) means that certain quantities~\(\theta(u,\xi)\) depending on decision variables and affected by random factors  should satisfy equality or inequality for ``almost all'' values of those random factors (according to their probability distributions). This is in particular the case of constraints which express ``laws of nature'' which are part of the mathematical model of the system under consideration. However, regarding objectives or ``wishes'', such strict constraints are generally inappropriate from the economic or simply realistic point of view. Suppose for example that a pressure should not exceed a certain level beyond which death will almost surely happen. First of all, observe that it is hard to aggregate such an objective (actually, that to stay alive) with other more economic objectives which aim at saving money. Second, under the constraint that the pressure ``almost never'' exceeds the dangerous level, the operation can be extremely costly if not simply impossible. That is, some \emph{risk} must be accepted for the operation to be economically viable.

The second possibility (constraints ``in expectation'') means that,
given a decision, the \emph{expected value} \(\Theta(u)=\EE \big(\theta(u,\xi)\big) \) of a critical quantity (a pressure in our example) should not
exceed a certain level. Such a formulation is generally mathematically
attractive, but it is difficult to understand how much risk is involved
in choosing such or such prescribed level. Indeed, given a decision $u$, the pressure (to keep on with our example) becomes a random variable \(\theta(u,\xi)\) with a certain distribution (which is affected by the chosen decision), and the only thing one ask is that the first moment (the expectation) of this random variable stay below a prescribed level, but with no direct control on how much of the probability mass will lie beyond that prescribed level.

The third possibility advocated to (constraints ``in probability'' or ``proba\-bilistic constraints'') means that one accepts that the critical quantity (the pressure, say) remains under the prescribed level not ``almost always'' as earlier, but with a certain probability whose value must be chosen. In mathematical terms, one now considers the problem
\begin{equation}
\label{eq-pbproba}
\min_{u\in\Uad}\EE \big(j(u,\xi)\big) \quad\text{s.t.}\quad \PP\big(\theta(u,\xi)\leq \alpha\big)\geq \ppp\;.
\end{equation}
This chosen probability value \(\ppp\) exactly reflects the risk one is ready to assume (in contrast with the previous approach of constraints in expectation). As discussed earlier, duality should then help in evaluating the sensitivity of the optimal cost function with respect to this accepted, but arbitrarily fixed, level of risk.

\subsection{Quantitative vs.~Qualitative Risk Measures}\label{sec-risk}
We now motivate the interest of probability constraints in contrast with other measures of risk.
Before choosing a risk measure, it is very important to
know which type of failure  we are interested in: quantitative
failure or qualitative failure. For example, a power supply company
would minimize its cost under the constraint of supplying the demand. If
that demand cannot be fully supplied, it matters to know which
percentage of it will not be covered  and during which amount of
time. This is what we mean by ``quantitative failure'': introducing a
penalty for the total amount of demand not supplied directly into the
cost function, or choosing to constrain a quantity which accounts for
the amount of supply failure is appropriate in that situation. On
the contrary, when  simply going beyond a certain threshold causes death,
it not does matter to know by which amount that threshold has been
exceeded --- this is what we mean by ``qualitative failure'' --- but it
does matter to know the likelihood of going beyond that critical
threshold. Probability constraints are particularly adapted to this
latter situation.

In fact, because of the mathematical
difficulties raised by probability constraints, these constraints must be
exclusively used in the case of qualitative failure problems. For quantitative failure problems, there are other risk measures with better
mathematical properties (e.g. convexity), like Conditional Value-at-Risk
(CVaR) for instance. Introduced in \cite{erUryasev}, CVaR is one of the
most popular risk measure in finance. CVaR is the average of a random
variable for the worst scenarios. Denote by $\alpha_{u}(\ppp)$ the quantile function of the distribution of \(\theta(u,\cdot)\) with
confidence level $\ppp$ (also called Value-at-Risk). Then, CVaR, denoted by
$\phi_{\ppp}(u)$, is defined by
$$\phi_{\ppp}(u)=\EE\big(\theta(u,\xi)\mid \theta(u,\xi)\geq \alpha_{ u}(\ppp)\big)\,.$$
The risk constraint will be then \(\phi_{\ppp}(u)\leq \overline{\alpha}\), where \(\overline{\alpha}\) represents the accepted level of risk, and \(\ppp\) is fixed a priori.

Notice that the critical threshold $\alpha$ in the probability constraint is generally provided by technical considerations, whereas \(\ppp\) characterizes the level of risk one is ready to accept. That is, the decision maker may bargain about the constraint level $\ppp$  but not on
that threshold $\alpha$ which is a technical data. With the CVaR approach, this \(\alpha\) disappears from the formulation and we believe that this is a
weakness of this approach. Moreover, in the case of ``qualitative failure'', there is no meaning in averaging values of \(\theta\) beyond a threshold which is supposed to be fatal.

\subsection{About this Paper}
Problem~\eqref{eq-pbproba} is the class of problems considered in this
paper. Its advantage is again the fact that the meaning of constraints
in terms of risk assumed is of immediate perception. Its drawback is its
mathematical difficulty.

In this paper, we discuss an approach relying upon Lagrangian duality
and stochastic gradient to solve \eqref{eq-pbproba}.
The use of stochastic gradient is based on the reformulation of
constraints  in probability as constraints in expectation, using an
indicator function. As usual with stochastic gradient, we assume that the functions involved in the problem (here, \(j\) and \(\theta\)) are known explicitly but that the probability law governing the ``noise''~\(\xi\) is not, or that the computation of expectations of the variables involved is out of reach or too costly. It is rather assumed that an external mechanism delivers samples of \(\xi\) which are used in the iterative algorithm. 

Writing the probability as an expectation opens the possibility of using
stochastic gradient algorithms, but it also raises the difficulty of
handling a discontinuous function, namely the indicator function. We will discuss various ways of overcoming that difficulty.

The rest of the paper is organized as follows. In
\S\ref{section2}, we present the analysis of the problem, and our
resolution strategy, a stochastic Arrow-Hurwicz algorithm. In
\S\ref{section3}, we describe two structural difficulties of
stochastic programming under probability constraint. To
implement a stochastic Arrow-Hurwicz algorithm, we need to handle
the probability function gradient. In~\S\ref{section4}, the
question we are interested in is therefore: how to compute
stochastic estimates of the probability function gradient? In
order to answer this question, we propose two methods that allow
to obtain \emph{biased} stochastic gradient estimates, namely Approximation by
Convolution (AC) ad Finite Differences (FD). We defer to a forthcoming paper to propose techniques based on integration by parts ideas and providing \emph{unbiased} (or \emph{consistent}) estimates, and to compare them with the biased estimates studied hereafter.

We consider a very basic
portfolio optimization problem under a probability constraint and use this example throughout the rest of the paper to illustrate and compare the AC and FD techniques. Section~\ref{section5} is devoted to the convergence analysis of the proposed methods. Finally, \S\ref{section6} reports numerical experiments with the Arrow-Hurwicz algorithm.

\section{Analysis of the Problem} \label{section2}

\subsection{Review of Main Difficulties} \label{difficulties}
Probability constraints provide a
straightforward risk formulation with an immediate intuitive
interpretation. But at the same time, it is well known that such
constraints raise important mathematical difficulties, such as the lack of
convexity or connectedness of the feasible  subset. Indeed, even if \(\theta\) is a convex function of \(u\) for almost all values
\(\xi\), the constraint in \eqref{eq-pbproba} may not define
a convex feasible subset in \(\UU\) (which can even be not connected, if not
empty). Those convexity or connectedness (or emptiness) properties depend
of course on the properties of \(\theta\) as a function of its two
arguments \(u\) and \(\xi\), on the probability distribution of the
random variable \(\xi\), on the level \(\alpha\) of constraint required
and on the level \(\ppp\) of probability required. One may refer to
\cite{erKall} for a discussion on those convexity properties, and to
\cite{erHenrion} for connectedness properties.

In \cite{erKall}, the authors prove that if $\theta(\cdot,\cdot)$ is
jointly convex in $(u,\xi)$ and the probability measure is quasi-concave, then the feasible
subset of \eqref{eq-pbproba} is convex. But those assumptions
seem to us to be rather strong in practice, notably the
joint convexity property. Indeed, there are numerous situations in which
the decision variable multiplies the random variable, as in the portfolio
problem presented in \S\ref{section3}, or in a quite other domain,
when one wants to model the breakdown of an actuator, in which case the random
variable must be able to kill the action the decision variable. In all those situations, the
joint convexity property is not realistic.

\subsection{Mathematical Approach for Programming under Probability Constraint}

Before explaining our resolution strategy, we review some basic results on the stochastic Arrow-Hurwciz algorithm \cite{arrow,culio}. First of all, starting with the deterministic constrained optimization problem~\eqref{eq-pb} and assuming that there exists a saddle point of the Lagrangian~\eqref{eq-lag} over \(\Uad\times \RR^{d}_{+}\), the (deterministic) Arrow-Hurwicz algorithm consists in performing successive minimization and maximization steps to search for this saddle point:
\begin{subequations}\label{eq-ahd}
\begin{align}
& u^{k+1}=\Pi_{U^{\mathrm{ad}}} \Big(
u^{k}-\varepsilon^{k}\big(\nabla_{u}J(u^{k})+\nabla_{u}\Theta(u^{k})\lambda^{k}\big) \Big)\,,\\
& \lambda^{k+1}=\Pi_{+}\Big(
\lambda^{k}+\rho^{k}\big(\Theta(u^{k+1}) - \alpha\big)\Big)\,,
\end{align}
where $\Pi_{U^{\mathrm{ad}}}$ is the projection onto $U^{\mathrm{ad}}$ and \(\Pi_{+}\) is the projection onto the cone \(\RR^{d}_{+}\).
\end{subequations}
\subsubsection{Stochastic Arrow-Hurwicz Algorithm} \label{algorithm}
The stochastic Arrow-Hurwicz algorithm is typically used to solve a
stochastic optimization problem with constraint in expectation:
\begin{equation}\label{eq-exp}
\min_{u \in \Uad}\EE \big( j(u,\xi) \big)
\quad\text{s.t.}\quad \EE \big(\theta(u,\xi)\big)\leq \alpha\,,
\end{equation}
where the calculation of expectations is basically difficult if not impossible. The stochastic algorithm overcomes this difficulty by simultaneously approximating the saddle point and
the expectations by a Monte-Carlo like technique. It is in fact a combination of the
idea of the Monte-Carlo method with the iterative procedure of gradient methods in optimization.

We do assume that a  saddle point \((u^{\sharp},\lambda^{\sharp})\) over \(\Uad\times \RR^{d}_{+}\)
exists for the Lagrangian associated with problem~\eqref{eq-exp} (hence \(u^{\sharp}\) is a solution of \eqref{eq-exp}). Observe that this Lagrangian~\(L\)~\eqref{eq-lag} is equal to the expectation of \(\ell(u,\lambda,\cdot)=j(u,\cdot)+\langle\lambda,\theta(u,\cdot)-\alpha\rangle\). We use unbiased estimates of the gradients of~\(L\) in \(u\) and \(\lambda\) obtained with the corresponding gradients of \(\ell\)  evaluated at independent drawings~\(\xi^{k}\) of \(\xi\) supposed to follow the probability law of \(\xi\).
More specifically, at stage $k$ of the algorithm, $u^{k}$ and $\lambda^{k}$ being the current estimates of the solution,
\begin{enumerate}
\item we draw an independent sample (according to the law $\PP$ of
  $\xi$), or we observe a new independent sample $\xi^{k+1}$,
\item we compute the stochastic gradients $
  \nabla_{u}j(u^{k},\xi^{k+1})$ and $ 
  \nabla_{u}\theta(u^{k},\xi^{k+1}),$
\item we update $u^{k+1}$ and $\lambda^{k+1}$ as follows:
\begin{subequations}\label{eq-ahs}
\begin{align}
& u^{k+1}=\Pi_{U^{\mathrm{ad}}} \Big(
u^{k}-\varepsilon^{k}\big(\nabla_{u}j(u^{k},\xi^{k+1})+\nabla_{u}\theta(u^{k},\xi^{k+1})\,\lambda^{k}\big) \Big)\,,\label{eq-ahs-a}\\
& \lambda^{k+1}=\Pi_{+}\Big(
\lambda^{k}+\rho^{k}\big(\theta(u^{k+1},\xi^{k+1})-\alpha\big) \Big)\,.\label{eq-ahs-b}
\end{align}
\end{subequations}
\end{enumerate}
Under essentially measurability and convexity assumptions, assuming that the
Lagrangian of the problem admits a saddle point, and with
$$\sum_{k \in \NN} \varepsilon^{k} = + \infty\,,\quad \sum_{k
  \in \NN}(\varepsilon^{k})^{2} < + \infty \qquad (\text{and the same
  for} \,\rho^{k})\,,$$
 it is shown in \cite{ercohen} that this algorithm converges in the sense that
  primal $\{u^{k}\}_{k \in \NN}$ and dual $\{\lambda^{k}\}_{k \in \NN}$
sequences are bounded a.s.~and that $\{u^{k}\}_{k \in \NN}$
a.s.~weakly converges to some solution~$u^{\sharp}$ of \eqref{eq-exp}.

\subsubsection{Mathematical Approach: Strategy and Difficulties}\label{sec-approach}
From now on, we assume that the critical or risky event is defined by a single (scalar) inequality, that is, \(\theta\) is \(\RR\)-valued. Let \(\II\) denote the indicator function of the positive half-line.
The principle of our resolution strategy is first to replace the
probability constraint by a constraint in expectation
\begin{equation}
\label{eq-exp2}
-P(u) \le -\ppp\,,
\end{equation}
where \(P(u)= \PP \big(\theta(u,\xi)\leq \alpha\big)\) and this probability is evaluated as an expectation:
\begin{equation}
\label{eq-probconstr}
 \PP \big(\theta(u,\xi) \le \alpha \big)= \EE
\Big( \II \big( \alpha-\theta(u,\xi)
\big) \Big)\,,
\end{equation}
then resort to duality, and lastly resort to
the stochastic Arrow-Hurwicz algorithm. Observe that w.r.t.~the general formulation~\eqref{eq-pb}, \(\Theta\) is now \(-P\) and the constraint level \(\alpha\) is now \(-\ppp\).

There are major difficulties with probability constraints.
\begin{itemize}
  \item First of all, as we recalled in \S\ref{difficulties},
convexity is not preserved. Therefore, existence of a saddle point of
the Lagrangian is not granted; in this case, we should resort to
\emph{augmented} Lagrangian techniques to increase the chance that a saddle point does exist. However, this raises new problems because the nonlinearities involved in the augmented Lagrangian formula cannot be combined straightforwardly with expectation to yield obvious stochastic  gradient algorithms. This issue of using augmented instead of ordinary Lagrangians goes beyond
the scope of this paper and is not considered here. 
\item We rather address here another difficulty: to replace a probability constraint by a
constraint in expectation, we need to handle the indicator function (see \eqref{eq-probconstr}); but this
indicator function involves a discontinuity which
 may, nevertheless,  be smoothed by the expectation operation; however, the
stochastic Arrow-Hurwicz algorithm is based on the consideration
of a \emph{unique} sample drawn at each iteration; obtaining a stochastic
gradient is therefore not trivial. As it will be shown later on,
we propose two ways of overcoming this difficulty: Approximation by Convolution (AC) and Finite Differences (FD). Both approaches will lead us to consider algorithms such as \eqref{eq-ahs} in which either a smooth  approximation \(\tapp\) of function~\(\theta\) will be used in both equations \eqref{eq-ahs-a} and \eqref{eq-ahs-b}, or an approximation of its gradient will be used in \eqref{eq-ahs-a}, leading to a stochastic Arrow-Hurwicz algorithm with \emph{biased} stochastic estimates of the Lagrangian gradients.
\end{itemize}
\section{Structural Difficulties of Programming under Probability Constraint} \label{section3}
In this section, we focus on two structural difficulties of optimization
problem with probability constraints. The first one is related to the
non convexity of probability constraint: we show what are the
consequences of this non convexity on the stochastic Arrow-Hurwicz
algorithm. The second one concerns the behavior of
the probability constraint multiplier in some particular cases.

\subsection{The Non Convexity of Probability Constraint}
Consider the following optimization problem
\begin{equation}
\label{eq-simple}
\min_{u \in \RR}\frac{1}{2}(u-1)^{2} \quad \text{s.t.} \quad
\PP ( u \le \xi)\ge \ppp\,,
\end{equation}
where $\xi$ is a normal random variable with mean value \(-2\) and standard deviation \(0.1\).

In order to point out the first difficulty, we study the qualitative
behavior of the underlying deterministic problem, namely that in which
the probability constraint is expressed with help of the
cumulative distribution function $F$ of $\xi$: indeed, \(\PP( u \le \xi )= 1-F(u)\) and therefore, the constraint in \eqref{eq-simple} can be replaced by
\begin{equation}
\label{eq-replace}
1-F(u) \ge \ppp
\end{equation}
without of course altering the corresponding Kuhn-Tucker multiplier.

The Lagrangian of problem~\eqref{eq-simple}, with constraint written as in \eqref{eq-replace}, is
$$L(u,\lambda)=\frac{1}{2}(u-1)^{2}+\lambda\,\big{(}F(u)-1+\ppp\big{)}\,;$$
and, the Kuhn-Tucker necessary conditions of optimality allow for the
calculation of the solution which is, for example with $\ppp=0.7$,
$$u^{\sharp}=-2.05244 \qquad \text{and} \qquad \lambda^{\sharp}=0.877913\,.$$
As expected, $u^{\sharp}$ takes the maximal possible value to satisfy the constraint, that is, the $(1-\ppp)$-th percentile of the distribution: $F(-2.05244) = 0.3$, so the constraint is active, and saturated.

Figure \ref{surf-lagt}
\begin{figure}[hbtp]
\begin{center}
\includegraphics[scale=0.5]{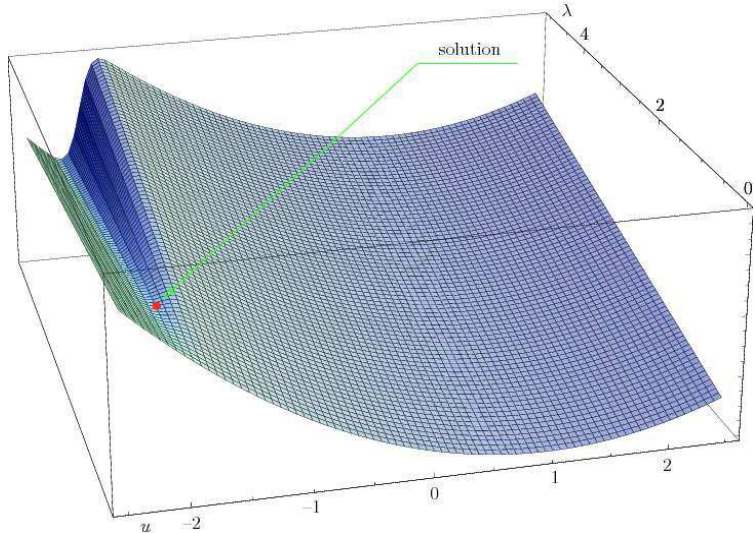}
\caption{Lagrangian surface}
\label{surf-lagt}
\end{center}
\end{figure}
represents the Lagrangian surface in the $(u,\lambda)$ domain.
For $\lambda=0$, we recognize the convex shape of
the cost function only. For larger values of $\lambda$, the nonconvex
form of $F(\cdot)$ shows up more and more, which explains the two valleys.

We insist on the following two points. First, our approach in this paper
is based on stochastic estimates of the gradients of the Lagrangian,
\emph{not} on their exact computation, which we assume impossible. Naturally, we cannot expect the stochastic algorithm to behave
better than its underlying \emph{average} driving vector field, which we will study directly.
Second, with some probability distributions there is a way to manipulate the constraints in order to preserve convexity. In particular with
normal distributions, the map $\ln(1-F(\cdot))$ is concave
\cite{erprekopa}, which leads to a convex formulation of the
constraint. If we seem to overlook this remark in the following
treatment, this is because the difficulty we try to point out in this
very simple case is {\it a fortiori\/} likely to occur in more general
situations  when the above clever manipulations  are no longer possible: recall that we do not assume knowledge of the distribution of $\theta(u,\xi)$.

Let us now consider the ODE
associated with the Arrow-Hurwicz algorithm,
\begin{subequations}
\label{eq-algo}
\begin{align}
 \dot{u}   & = - L'_{u}(u,\lambda)= -J'(u)-\lambda\,F'(u)\;,  \\
  \dot{\lambda}  & = L'_{\lambda}(u,\lambda) =F(u)-1+\ppp\;.
\end{align}
\end{subequations}
At $u^{\sharp}=1$, the \emph{unconstrained} optimal solution, one has
that $J'(u^{\sharp})=0$ and $F'(u^{\sharp})=\mbox{\(1.47 \times 10^{-195}\)}$, because
$u^{\sharp}$ happens to be in the tail of the distribution. Therefore, even for
very large values of $\lambda$, $L_{u}'(u^{\sharp},\lambda) $ remains very close to~0; in
other words, if the (continuous) algorithm~\eqref{eq-algo} is started at (or close to) \((u^{\sharp},\lambda)\), for practically any \(\lambda\), $u$ will stay at $u^{\sharp}$! At the same time, if
$u^{\sharp}$ doesn't satisfy the constraint, one has that
$F(u^{\sharp})>1-\ppp$. It follows that $L_{\lambda}'(u^{\sharp},\lambda)>0$:
$\lambda$ increases almost indefinitely! This is better illustrated by
the vector field of the ODE, shown in Figure \ref{champ}.
\begin{figure}[hbtp]
\begin{center}
\includegraphics[scale=0.4]{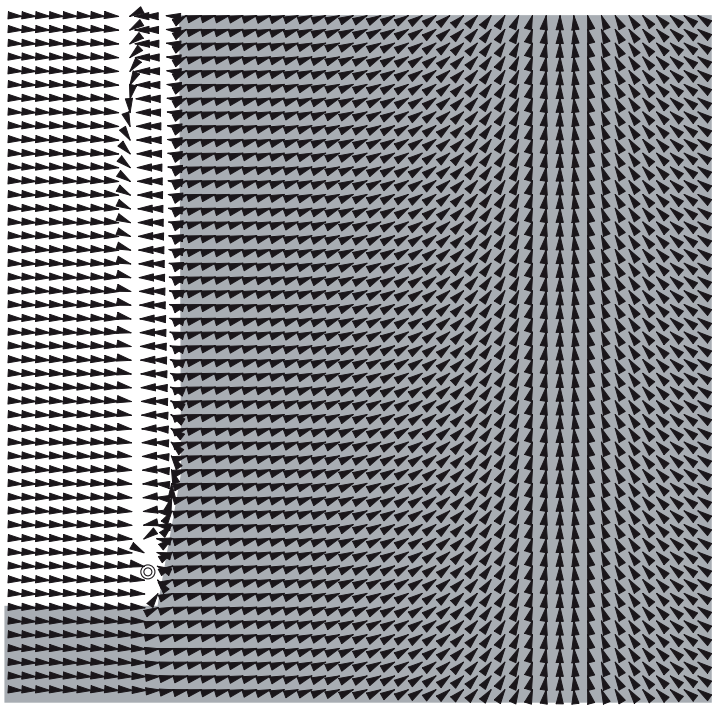}
\caption{Vector field of the ODE}
\label{champ}
\end{center}
\end{figure}

The white zone corresponds to the basin of
attraction of the optimal solution. In the grey zone, the algorithm is
driven more or less indefinitely towards large values of $\lambda$ in a
valley corresponding to the unconstrained solution $u^{\sharp}=1$.

This example shows that even a deterministic algorithm may, if started on the ``wrong'' side, wander away from the actual optimal solution. Stochastic versions of the algorithm are expected to behave erratically, and even if the current values of $(u^k,\lambda^k)$ are in the basin of attraction of the Kuhn-Tucker point, random observations may take the algorithm to other regions away from the optimal solution.

\subsection{Degeneracy of the Probability Constraint Multiplier}\label{sec-example}
Let us now consider the following portfolio optimization problem. This
very simple problem allows us to point up another structural difficulty
of probability constraint. This example will also be used in the remainder of this paper to illustrate our various approaches.

We borrow a capital which we have to pay off at the end of the
period with an interest rate~$l$. We can invest a proportion $u$ of
this capital at the fixed rate $b$, invest a proportion $v$ at the random rate $\xi$, and finally consume the
available remainder, which brings a satisfaction measured by
a concave nondecreasing function $f$. We assume of course that $\EE ( \xi )
>l$, in other words, risk is rewarding. We try to maximize the sum of the
  satisfaction provided by consumption and by the expected final
  capital. We also want to be in a position to pay off
  the capital and the
interests at the end of the period, with a probability of a least
$p$. In this case, the optimization problem can be stated as follows:
\begin{align*}
& \max_{u,v} \EE \big(f(1-u-v)+(1+b)u+(1+\xi)v \big)\\
\text{s.t.} \quad & u \ge 0\,,\quad v \ge 0\,,\quad u+v \le 1\,,\\
 & \PP \big((1+b)u+(1+\xi)v \ge 1+l \big)\ge \ppp\,.
\end{align*}
Let
\begin{gather}
l=0.15, \qquad b=0.2, \qquad f(x)=-x^{2}/2+2x\,,\notag\\
F(\xi)=\begin{cases}
       0 & \qquad \mbox{if} \qquad \xi < \bar{\xi}-\sigma\,,\\
        \frac{1}{16} \Big( 3\big(
       \frac{\xi-\bar{\xi}}{\sigma} \big)^{5}-10\,\big(
       \frac{\xi-\bar{\xi}}{\sigma} \big)^{3}+15\,\big(
       \frac{\xi-\bar{\xi}}{\sigma} \big)+8 \Big) & \qquad \mbox{if} \qquad  \xi
       < \bar{\xi}+\sigma\,,\\
       1 &\qquad \mbox{otherwise}\,,
\end{cases}\label{eq-F}
\end{gather}
where $F$ is the distribution function. For numerical experiments, we
set $\bar{\xi}=0.4$ and $\sigma=3$. To identify this problem with \eqref{eq-pbproba},  consider the equivalent
minimization problem with cost function 
$$j(u,v,\xi)=-f(1-u-v)-(1+b)\,u-(1+\xi)\,v\,.$$ Let also
\begin{equation}P(u,v)=\PP \big ((1+b)\,u+(1+\xi)\,v \ge 1+l
\big)\,.\label{eq-proo}\end{equation}
This problem is now formulated as
\begin{displaymath}
\min_{u\geq 0, v\geq 0} \EE\big(j(u,v,\xi)\big)\quad \text{s.t.}\quad u+v \leq 1, \quad -P(u,v)\leq -\ppp\,
\end{displaymath}
with Lagrangian
\begin{displaymath}
L(u,v;\lambda_{1},\lambda_{2})=\EE\big(j(u,v,\xi)\big)+\lambda_{1}(u+v-1)+\lambda_{2}(\ppp-P(u,v))\,.
\end{displaymath}

Figure~\ref{coutoptimal} represents the \emph{optimal cost} as a function of
probability level $\ppp$.
\begin{figure}[hbtp]
\begin{center}
\includegraphics[scale=1]{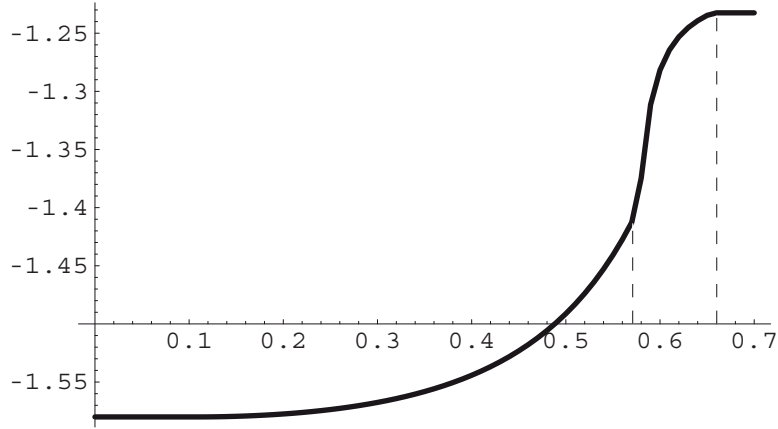}
\caption{Optimal cost}
\label{coutoptimal}
\end{center}
\end{figure}
We observe that this function is not convex. In fact, it is convex for
probability levels below 0.57. For probability levels close to 0.57, the
risk of not being in a position to pay off the capital and the interests
is important; the investment in the risky asset $v$ decreases to zero, whereas
simultaneously, that in the secure asset $u$
increases. The optimal cost, which was until then a convex function of
the required probability level, becomes concave. Above 0.65, $v$
is zero, $u$ is equal to $(1+l)/(1+b)=0.95833$ in order to satisfy the probability constraint, and the optimal cost
becomes constant. 

This example shows  another structural difficulty of optimization under a
probability constraint, namely the degeneracy of the probability
constraint multiplier. Indeed, for $\ppp$ small enough, the secure asset,
$u$ is zero at optimum. For $\ppp$ large enough, the risky asset,
$v$, is zero at optimum. In the latter case, the event
$\big{\lbrace} (1+b)\,u+(1+\xi)v \geq 1+l\big{\rbrace}$ can only have
probability 0 or 1. That is to say, at $v=0$ and $u=(1+l)/(1+b)=0.95833$, the
function \eqref{eq-proo}
exhibits a  discontinuity (see Figure~\ref{discontinut}).
\begin{figure}[hbtp]
\begin{center}
\includegraphics[scale=1]{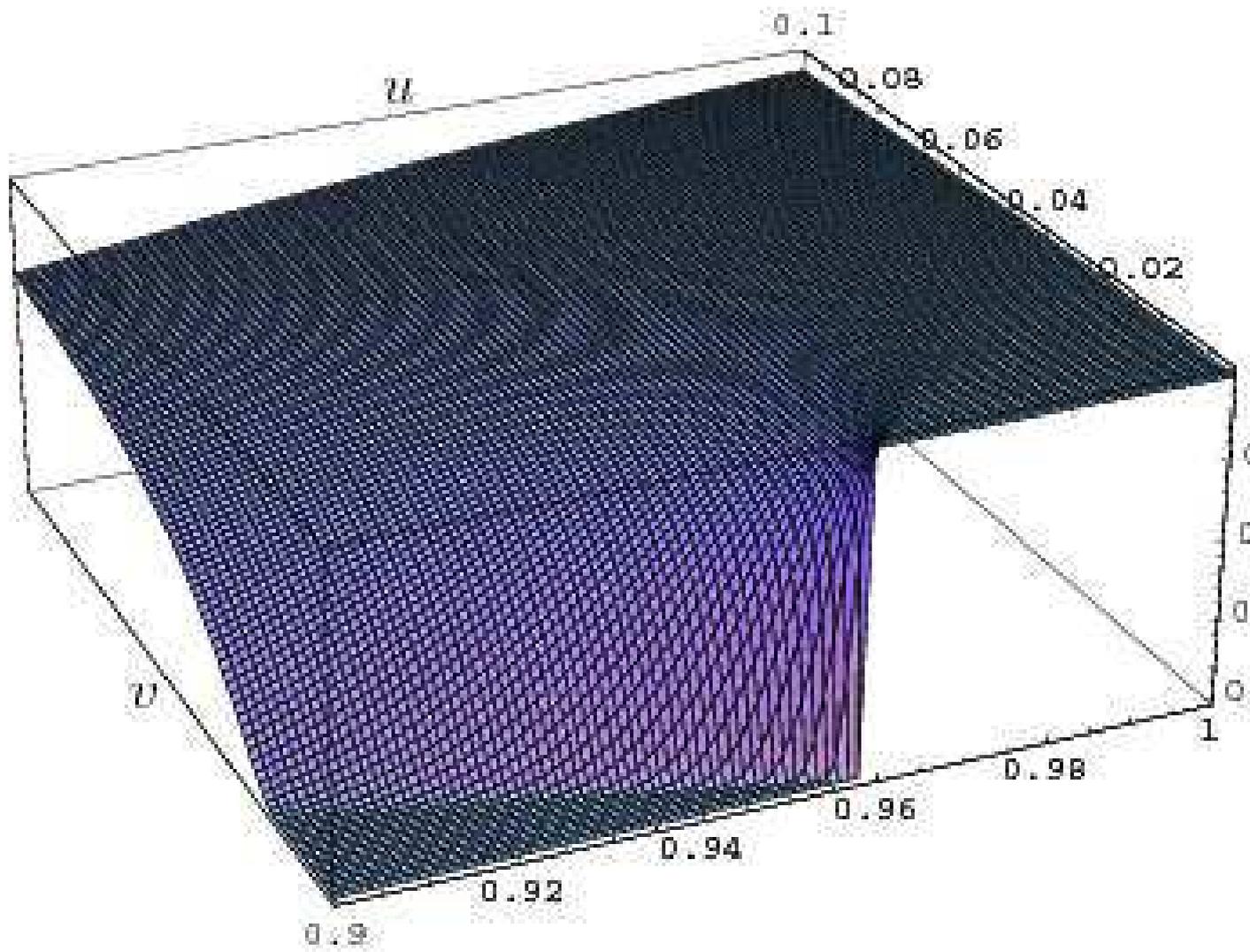}
\caption{Graph of the probability function for \((u,v)\in[0.9,1]\times[0,0.1]\)}
\label{discontinut}
\end{center}
\end{figure}

When  $\ppp$ is large and  $v^\sharp =0$, the probability~\eqref{eq-proo} can only take values 0 or 1 (depending on the value of \(u\)), that is, this probability is strictly larger than the required level \(\ppp\) when the constraint is met. Clearly the constraint is not ``saturated'', because there is no equality, and consequently the corresponding multiplier is zero; small changes in $\ppp$ will not affect the solution. However, the constraint is ``active'', that is,   the solution~\((u^{\sharp},v^{\sharp})\) of the problem \emph{with} the probability constraint is different from the solution~\((u^{*},v^{*})\) \emph{without} it. 

In the remainder, in order to guarantee existence of a saddle point of the Lagrangian, we consider only probability levels below 0.57 (otherwise,
one should resort to augmented Lagrangian techniques, but, as it has been said
earlier, this issue goes beyond the scope of this paper). For example, for a probability level of 0.24, the primal-dual optimal solution is
\begin{equation}
\label{eq-sharp}
u^{\sharp}=0\,,\quad v^{\sharp}=0.50407\,,\quad \lambda_{1}^{\sharp}=0\,, \quad \lambda_{2}^{\sharp}=0.08815 \,.
\end{equation}

\section{Stochastic Estimates of Probability Function Gradient}
\label{section4}
As it was mentioned at \S\ref{sec-approach}, in order to use a stochastic
Arrow-Hurwicz algorithm, we need to handle the probability function
gradient, that is, to obtain a stochastic estimate of the gradient of  (see \eqref{eq-probconstr})
\begin{equation}
\label{eq-psi}
P(u)= \EE \Big(
\II\big(\alpha-\theta(u,\xi)\big) \Big)\,.
\end{equation}
It is well known that this gradient is difficult to compute; we
may refer to \cite{eruryasev} for a discussion of this topic.  Recall that in our case, replacing the probability
constraint by a constraint in expectation raises the difficulty of handling
an indicator function, which is a discontinuous function.
One way of
dealing with this problem is to appeal to a technique based on convolution to derive a smooth approximation of this discontinuous function. Alternatively, we can obtain a stochastic estimate of the gradient of this function, based on a single sample drawing of \(\xi\), by appealing to a finite difference technique, and we rely upon the multiplication of such drawings along the iterative algorithm to smooth out this crude estimate.

\subsection{Approximation by Convolution Method (AC)}
\subsubsection{General Theory}\label{sec-gen}
The basic principle of this approach is to smooth out the indicator function appearing in \eqref{eq-psi} so that differentiation underneath expectation becomes possible. Consider a smooth function~\(h: \RR \rightarrow \RR\) with the following properties~: \(h\) as a unique maximum at \(x=0\),
\begin{equation}
\label{eq-prop}
\forall x,\quad h(x)\geq 0;\quad h(x)=h(-x);\qquad \int_{-\infty}^{+\infty}h(x)\,dx = 1\,.
\end{equation}
We will give a few examples of such functions later on and will consider only functions with finite support although this is not absolutely necessary. With any other function~$\phi : \RR \rightarrow \RR$, and \(r\) a small positive number, the
convolution
$$\phi_{r}(x)=\frac{1}{r}\, \int_{-\infty}^{\infty} \phi(y) h \Big(
\frac{x-y}{r} \Big)\,dy\,,$$
can be viewed as an approximation of $\phi$ since \(h(\cdot/r)/r\) approximates the Dirac function (in the sense of convergence of distributions) at \(0\) when \(r\) tends to
zero. The function \(\phi_{r}\) is differentiable with
$$\phi_{r}'(x) =\frac{1}{r^{2}}\int_{-\infty}^{\infty} \phi(y)\,h'\Big(\frac{x-y}{r}\Big)\,dy\,.$$

This technique is widely known as the ``mollifier'' technique \cite{erErmoliev}. We now apply it to \(\II\): recall \eqref{eq-psi} and define
\begin{align}
P_{r}(u) &= \frac{1}{r}\EE\Bigg(\int_{-\infty}^{+\infty}
\II(y)\, h\Big(\frac{\alpha-\theta(u,\xi)-y}{r}\Big) \,dy\Bigg)\notag\\&= \frac{1}{r}\EE\Bigg(\int_{0}^{+\infty}
h\Big(\frac{y-\alpha+\theta(u,\xi)}{r}\Big) \,dy\Bigg)\,\notag\\
\intertext{(here, we have used the fact that \(h\) is an even function)}
&=\EE\big(p_{r}(u,\xi)\big)\label{eq-tha}\\
\intertext{with}
p_{r}(u,\xi)&= \frac{1}{r}\int_{0}^{+\infty}
h\Big(\frac{y-\alpha+\theta(u,\xi)}{r}\Big) \,dy\,.\label{eq-ther}
\end{align}
Then 
\begin{align}
(p_{r})'_{u}(u,\xi)& = \frac{1}{r^{2}}\,\theta'_{u}(u,\xi)\int_{0}^{+\infty}
 h'\Big(\frac{y-\alpha+\theta(u,\xi)}{r}\Big) \,dy\notag\\&=-\frac{1}{r}\,h\Big(\frac{\theta(u,\xi)-\alpha}{r}\Big)\,\theta'_{u}(u,\xi)\,,\label{eq-bias}\\\intertext{and clearly}
 P'_{r}(u)&=\EE\big((p_{r})'_{u}(u,\xi)\big)\,.\label{eq-bias1}
\end{align}
Therefore, for any sample \(\xi\), \eqref{eq-bias}
can be considered as a stochastic estimate of \(P'(u)\), albeit a \emph{biased} one; however, this bias vanishes when \(r\) approaches~0. In what follows, we evaluate the bias and the variance of this estimate as a function of \(r\).
\begin{remark}\label{rem-variante}
In the same way, according to \eqref{eq-tha}, \eqref{eq-ther} can be considered a biased estimate of \(P(u)\) whereas
\begin{equation}
\label{eq-puxi}
p(u,\xi)=\II\big(\alpha-\theta(u,\xi)\big)
\end{equation}
is an unbiased one. In Equation~\eqref{eq-ahs-b} of the iterative algorithm, we may either use the unbiased estimate or the biased one, consistently with that used in \eqref{eq-ahs-a} for \(\Theta'\). The latter option has the advantage of preserving the specific geometry of vector fields of Arrow-Hurwicz algorithms (with some symmetry, or skew-symmetry, properties, according to the point of view). The former option may seem preferable as long as it avoids seemingly unnecessary bias or approximation. Both options will be tested later on in \S\ref{section6}. Therefore, the next theorem deals with both the estimates~\eqref{eq-ther} and \eqref{eq-bias} in order to cover all variants.
\end{remark}

\begin{theorem}\label{th-1}The random variable (or vector)~\(\xi\) is supposed to admit a density \(q(\xi)\). For the random variable \(X_{u}(\cdot)=\theta(u,\cdot)\) depending on the parameter~\(u\), we assume that the induced probability law also admits a density denoted \(q_{X_{u}}(x)\) and that this density is at least twice continuously differentiable with \(L^1\) first and second order derivatives. Then, for any sample drawing \(\xi\) following the probability density \(q\), the expression~\eqref{eq-ther}  provides a biased estimate of \(P(u)\) with a bias in \(\mathrm{O}(r^{2})\) and a variance in \(\mathrm{O}(1)\).

For the pair of random variables (or vectors) \(\big(X_{u}(\cdot),Y_{u}(\cdot)\big)=\big(\theta(u,\cdot),\theta'_{u}(u,\cdot)\big)\) depending on the parameter~\(u\), we assume that the induced joint probability law admits a density denoted \(q_{X_{u}Y_{u}}(x,y)\) and that this density is at least twice continuously differentiable in \(x\) with integrable \(L^1\) first and second order derivatives.
Then, for any sample drawing \(\xi\) following the probability density \(q\), the expression~\eqref{eq-bias} provides a biased estimate of \(P'(u)\) with a bias in \(\mathrm{O}(r^{2})\) and a variance in \(\mathrm{O}(1/r)\).
\end{theorem}

\begin{proof}
Consider first \eqref{eq-tha}--\eqref{eq-ther}. With the induced probability law for the random variable \(X_{u}\), one has that
\begin{displaymath}
P_{r}(u)=\frac{1}{r}\int_{-\infty}^{+\infty}\int_{0}^{+\infty}h\Big(\frac{y-\alpha+x}{r}\Big)\, q_{X_{u}}(x)\,dy\,dx\,.
\end{displaymath}
Using Fubini theorem and the change of variable \(z=(y-\alpha+x)/r\) in the integral in \(x\) yields
\begin{displaymath}
P_{r}(u)=\int_{0}^{+\infty}\int_{-\infty}^{+\infty}h(z)\, q_{X_{u}}(rz-y+\alpha)\,dz\,dy\,.
\end{displaymath}
With the smoothness assumptions on \(q_{X_{u}}\), the Taylor expansion of this term for \(r\) near 0 yields
\begin{align*}
P_{r}(u)&=\int_{0}^{+\infty}\int_{-\infty}^{+\infty}h(z)\, \Big(q_{X_{u}}(\alpha-y)+r z\, q'_{X_{u}}(\alpha-y)+\frac{r^{2} z^{2}}{2}\,q''_{X_{u}}(\alpha-y)+\mathrm{O}(r^3)z^{3}\Big)\,dz\,dy\\
&=\int_{0}^{+\infty}q_{X_{u}}(\alpha-y)\,dy+\frac{r^{2}}{2}\sigma^2_{h}\int_{0}^{+\infty}\,q''_{X_{u}}(\alpha-y)\,dy+\mathrm{O}(r^3)
\end{align*}
by using \eqref{eq-prop} on the one hand and by introducing
\begin{equation}
\label{eq-sh2}
\sigma^2_{h}=\int_{-\infty} ^{+\infty}z^2\, h(z)\,dz
\end{equation}
on the other hand. The term of order 0 in \(r\) can be written as 
\begin{displaymath}
\int_{-\infty}^{\alpha} q_{X_{u}}(t)\,dt
\end{displaymath}
and, as such, is recognized to be equal to \(\PP(X_{u}\leq \alpha)\), that is, \(P(u)\). Therefore, \(P_{r}(u)\) differs from \(P(u)\) by an \(\mathrm{O}(r^{2})\) term (proportional to \(\sigma^2_{h}\)). 

As for the variance of the estimate~\eqref{eq-ther}, it is equal to the second order moment \(\EE\Big(\big(p_{r}(u,\xi)\big)^{2}\Big)\) from which the square of \(\EE\big(p_{r}(u,\xi)\big)\) must be subtracted. The latter is close to \(\big(P(u)^{2}\big)\) up to a term of order \(\mathrm{O}(r^{2})\). Therefore we concentrate on the second order moment which can be written, according to \eqref{eq-ther},
\begin{align*}
\EE\Big(\big(p_{r}(u,\xi)\big)^{2}\Big)&= \frac{1}{r^{2}} \EE\Bigg(\Big(\int_{0}^{+\infty}
h\Big(\frac{y-\alpha+\theta(u,\xi)}{r}\Big) \,dy\Big)^{2}\Bigg) \\
&= \frac{1}{r^{2}} \int_{-\infty}^{+\infty}\Big(\int_{0}^{+\infty}
h\Big(\frac{y-\alpha+x}{r}\Big) \,dy\Big)^{2}\,q_{X_{u}}(x)\,dx\\
&= \int_{-\infty}^{+\infty}\Big(\int_{\frac{x-\alpha}{r}}^{+\infty}
h(z) \,dz\Big)^{2}\,q_{X_{u}}(x)\,dx\\\intertext{using the change of variable \(z=(y-\alpha+x)/r\) in the integral in \(y\),}&\leq\int_{-\infty}^{+\infty}\Big(\int_{-\infty}^{+\infty}
h(z) \,dz\Big)^{2}\,q_{X_{u}}(x)\,dx\\\intertext{since \(h(\cdot)\geq 0\),}&=
1
\end{align*}
according to \eqref{eq-prop} (last equality) and the fact that \(q_{X_{u}}\) is a probability density. 

The proof regarding the bias of \(P'_{r}(u)\) w.r.t.~\(P'(u)\) may follow one of the following two paths: either a similar result is proved for the derivative of a function whenever the function itself is approximated by another function up to an \(\mathrm{O}(r^{2})\) term; or, with \eqref{eq-bias}--\eqref{eq-bias1}, we perform similar calculations to those we have just performed with \eqref{eq-tha}--\eqref{eq-ther}. Let us sketch this second path. Considering \eqref{eq-bias}--\eqref{eq-bias1} and the pair \(\big(X_{u}(\cdot), Y_{u}(\cdot)\big)\), we have that
\begin{displaymath}
P'_{r}(u)=-\frac{1}{r}\int\int h\Big(\frac{x-\alpha}{r}\Big)\,y\, q_{X_{u}Y_{u}}(x,y)\,dx\,dy
\end{displaymath}
(remember \(y\) may be a vector of the same dimension as \(u\) and \( d y\) should be understood adequately). From here, we proceed as previously with the change of variable \(z=(x-\alpha)/r\) which yields
\begin{displaymath}
P'_{r}(u)=-\int\int h(z)\,y\,q_{X_{u}Y_{u}}(rz +\alpha,y)\,dz\,dy\,.
\end{displaymath}
Then, a Taylor expansion of \(q_{X_{u}Y_{u}}\) w.r.t.~its first argument for  \(r\) near~0 yields, for the same reasons as previously,
\begin{equation}
P'_{r}(u)=-\int y\,q_{X_{u}Y_{u}}(\alpha,y)\,dy+\frac{r^2}{2}\sigma^2_h \int \frac{\partial^{2} q_{X_{u}Y_{u}}(\alpha,y)}{\partial x^{2}}\, y \, dy+ \mathrm{O}(r^{3})\,.\label{eq-question}
\end{equation}
Assuming that the first term in the right-hand side above is equal to \(P'(u)\) (see Claim~\ref{claim} hereafter), we obtain again that \(P'_{r}(u)\) differs by an \(\mathrm{O}(r^{2})\) term. 

 The variance is equal to the second order moment \(\EE\Big(\big(p'_{r}(u,\xi)\big)^{2}\Big)\) from which we must subtract \(\Big(\EE\big(p'_{r}(u,\xi)\big)\Big)^{2}\). The latter is close to \(\big(P(u)\big)^{2}\) up to \(\mathrm{O}(r^{2})\). As for the former, we have that
 \begin{align*}
\EE\Big(\big(p'_{r}(u,\xi)\big)^{2}\Big)&= \frac{1}{r^{2}}\int h^{2}\Big(\frac{\theta(u,\xi)-\alpha}{r}\Big)\big(\theta'_{u}(u,\xi)\big)^{2} q(\xi)\,d\xi\\
&=\frac{1}{r^{2}}\int \int h^{2}\Big(\frac{x-\alpha}{r}\Big)\,y^{2} q_{X_{u}Y_{u}}(x,y)\,dx\,dy\\
&=\frac{1}{r}\int \int h^{2}(z)\,y^{2} q_{X_{u}Y_{u}}(r z+\alpha,y)\,dz\,dy\,.
\end{align*}
From here, we proceed as earlier with a Taylor expansion for \(r\) close to~0, and it should be clear that the above expression is of order~\(1/r\) with a coefficient which can be bounded by a term proportional to the square of the \(L^{2}\) norm of \(h\). The same consideration is still valid for the variance itself.
\end{proof}

\begin{claim}\label{claim}
The first term in the right-hand side of \eqref{eq-question} is equal to \(P'(u)\). We sketch the proof of this fact here. For any smooth function \(f: \RR\to \RR \), consider
\begin{displaymath}
F(u)= \EE\Big(f\big(\theta(u,\xi)\big)\Big)= \int f\big(\theta(u,\xi)\big)\,q(\xi)\,d\xi=\int f(x)\,q_{X_{u}}(x)\,dx\,.
\end{displaymath}
Then,
\begin{displaymath}
F'(u)=  \int f'\big(\theta(u,\xi)\big)\,\theta'_{u}(u,\xi)\,q(\xi)\,d\xi=\int\int f'(x)\,y\,q_{X_{u}Y_{u}}(x,y)\,dx\,dy\,.
\end{displaymath}
Integrating by parts in the integral in \(x\), one gets
\begin{displaymath}
F'(u)=  -\int\int f(x)\,y\,\frac{\partial q_{X_{u}Y_{u}}}{\partial x}(x,y)\,dx\,dy\,.
\end{displaymath}
If \(f\) is not smooth enough for this calculation to be immediately justified, one can consider a sequence of smooth approximations converging to \(f\) in order to establish this formula. Let us now use it for \(f(\cdot)= \II(\alpha-\cdot)\). Then \(F(u)=P(u)\) and
\begin{align*}
P'(u)&= -\int y\int \II(\alpha-x)\,\frac{\partial q_{X_{u}Y_{u}}}{\partial x}(x,y)\,dx\,dy\\& =-\int y\int_{-\infty}^{\alpha} \frac{\partial q_{X_{u}Y_{u}}}{\partial x}(x,y)\,dx\,dy\\
&=-\int y \,q_{X_{u}Y_{u}}(\alpha,y)\,dy\,,
\end{align*}
which is exactly the expected result.
\end{claim}
\begin{remark}\label{rem-rem}
As a side remark,
observe that
\begin{displaymath}
q_{X_{u}}(\alpha)=\int q_{X_{u}Y_{u}}(\alpha,y)\, dy\,,
\end{displaymath}
and that
\begin{displaymath}
q_{X_{u}Y_{u}}(\alpha,y)/q_{X_{u}}(\alpha)=q_{Y_{u}}(y\mid X_{u}=\alpha)\,,
\end{displaymath}
that is, the conditional density of \(Y_{u}\) knowing that \(X_{u}=\alpha\). Therefore, the first term in the right-hand side of \eqref{eq-question} can be written as \(-\EE\big(Y_{u}\mid X_{u}=\alpha\big)\times q_{X_{u}}(\alpha)\). We conclude that 
\begin{displaymath}
P'(u)=- q_{\theta(u,\cdot)}(\alpha)\times \EE\big(\theta'_{u}(u,\cdot)\mid \theta(u,\cdot)=\alpha\big) \,.
\end{displaymath}
\end{remark}

\begin{remark}
Observe that, although we started with the idea of a smooth function~\(h\), the expression~\eqref{eq-bias} of the estimate and the analysis in the proof of Theorem~\ref{th-1} does not involve more than the function \(h\) itself (not its derivatives), so that we may as well consider non smooth functions (and even discontinuous functions at the ends of its support).
\end{remark}

In conclusion, the variance of the stochastic estimate~\eqref{eq-bias} blows up like \(A/r\) as \(r\) goes to 0 (where \(A\) can be bounded from above by something proportional to the square of the \(L^{2}\) norm of \(h\)), that of \eqref{eq-ther} remains of order \(\mathrm{O}(1)\), whereas the square of the bias of both estimates goes to 0 as \(B^2 r^{4}\) (where \(B\) is proportional to \(\sigma^{2}_{h}\) --- see \eqref{eq-sh2}). If the estimate of \(P'(u)\) is rather based on the average of \(N\)~expressions as \eqref{eq-bias} for \(N\)~independent drawings of \(\xi\), the variance will blow up as \(A/(N r)\)  whereas the square of the bias will still behave as \(B^2 r^{4}\). Therefore, the best trade-off between variance and bias is realized by that \(r\) which minimizes the mean square error (MQE; this is the sum of the variance and of the square of the bias) equal to \(A/(Nr)+ B^2 r^{4}\): the ``best'' \(r\) is thus \(\big(A/(4 B^2 N)\big)^{1/5}\). This yields a MQE estimated to \((5 A^{4/5}B^{2/5})/(4 N)^{4/5}\). Therefore, in the choice of function~\(h\), it is meaningful to pay attention to the quantity \(\sigma^{4/5}_{h}\|h\|_{L^2}^{8/5}\). Remember \(B\) is proportional to \(\sigma^{2}_{h}\) and \(A\) is proportional to \(\|h\|^{2}_{L^{2}}\).

The bias of the AC estimate goes to 0 with \(r\): this parameter~\(r\)
allows for a trade-off between mean and variance which should be adapted to the number of samples available (as just discussed) or visited in one run in the context of an iterative algorithm, as discussed later on in \S\ref{sec-rate}.

\subsubsection{Practical Aspects and Application to Example of \S\ref{sec-example}}\label{sec-sec}
Define
\begin{displaymath}
I(x)=\begin{cases}
   1   & \text{if} -1\leq x \leq 1, \\
    0  & \text{otherwise}.
\end{cases}
\end{displaymath}
Table~\ref{tab-funct} proposes 6~functions with bounded supports that can play the role of function~\(h\) (see \eqref{eq-prop}) and compares them from the point of view of their constants \(\sigma^{4/5}_{h}\|h\|_{L^2}^{8/5}\) (last column), the relevance of which has just been explained. The column \(h(0)\) is provided to help identifying the functions with their graphs displayed in Figure~\ref{fig-mollifunct}.
\begin{table}[hbtp]
$$
\begin{array}{|c|c|c|c|c|}\hline
\rule[-2ex]{0pt}{4ex}h&h(0)&\sigma^2_{h}&\|h\|^2_{L^2}&\sigma^{4/5}_{h}\|h\|_{L^2}^{8/5}\\\hline
\hline I(x)&0.5000&0.3333&0.5000&0.3701\\
\hline (1-|x|)I(x)&1.000&0.1667&0.6667&0.3531\\
\hline \pi\cos(\pi x/2)I(x)/4&0.7854&0.1894&0.6169&0.3492\\
\hline 3(1-x^2)I(x)/4&0.7500&0.2000&0.6000&0.3491\\
\hline 15(1-x^2)^2 I(x)/16&0.9375&0.1429&0.7143&0.3508\\
\hline 35(1-x^2)^3 I(x)/32&1.0938&0.1111&0.8159&0.3529\\\hline
\end{array}
$$\caption{Various \(h\) functions}\label{tab-funct}\end{table}
\begin{figure}[hbtp]
\centering
 \includegraphics{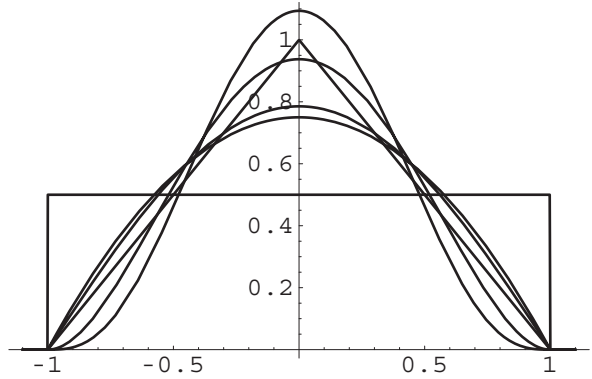}
 \caption{Several possible \(h\) functions}
\label{fig-mollifunct}
\end{figure}

We observe that the fourth function, namely \(h(x)=3(1-x^2)I(x)/4\) is the one to retain because it offers the smallest value in the last column of the table. We now apply the technique to the example of \S\ref{sec-example} again. The estimates for \(P'_{u}(u,v)\) and \(P'_{v}(u,v)\) based on this technique and on a given sample~\(\xi\) read as follows:
\begin{subequations}\label{eq-estuv}
\begin{align}
(p_{r})'_{u}(u,v,\xi)=\frac{1+b}{r}\,h\Big(\frac{(1+\alpha)-(1+b)u-(1+\xi)v}{r}\Big)\,;\\
(p_{r})'_{v}(u,v,\xi)=\frac{1+\xi}{r}\,h\Big(\frac{(1+\alpha)-(1+b)u-(1+\xi)v}{r}\Big)\,.
\end{align}
\end{subequations}
This MQE will be compared with that obtained by the next approach, namely finite differences.

We performed some exact computations of bias and variance with the help of \emph{Mathematica} for those estimates evaluated at the optimal solution (see \eqref{eq-sharp}) and with \(h\) equal to the fourth function in Table~\ref{tab-funct}. We have found:
\begin{align*}
\EE(p_{r})'_{u}(u^{\sharp},v^{\sharp},\cdot)& =0.62 - 
    0.096 r^2 + 0.012 r^4\,,\\
\var(p_{r})'_{u}(u^{\sharp},v^{\sharp},\cdot)  & = 0.45/r-0.39 - 
      0.05 r + \mathrm{O}(r^{2})\,,\\
\EE(p_{r})'_{v}(u^{\sharp},v^{\sharp},\cdot)& = 1.18 - 
    0.36 r^2 + 0.06 r^4\,,\\
\var(p_{r})'_{v} (u^{\sharp},v^{\sharp},\cdot) & =1.62/r -1.39  - 
      0.35 r + \mathrm{O}(r^{2})\,.
\end{align*}
If the estimates are based on the average over \(N\)~samples, the MQE of the AC estimates are obtained by considering \(\var(r)/N +(\EE(r))^2 -(\EE(0))^2\). In those expressions, we consider the terms in \(1/Nr\) and \(r^4\) \emph{only} in order to tune \(r\) as a function of \(N\). This computation is done for the sum of the MQE's related to the two components of \(p_{r}'\) (that is, for the mean square norm of the vector estimate error --- we denote it \(\mathrm{MQE}(r,N)\)). This yields \(r= 1.30\, N^{-1/5}\). Finally, we plug this value of \(r\) into \(\mathrm{MQE}(r,N)\) to get the following function of \(N\) (again, calculations are exact, up to \emph{Mathematica} accuracy, even if results are displayed in a truncated form):
\begin{equation}
 \frac{1.98}{N^{4/5}}-\frac{1.78}{N}+\mathrm{O}(N^{-6/5})\;.\label{eq-mqe}
\end{equation}

\subsection{Finite Differences (FD)}
\subsubsection{General Theory}
The idea here is simply to evaluate the derivative w.r.t.~each component \(u_{j}\) of the expression inside expectation in \eqref{eq-psi} by the  variation of this quantity, caused by, and divided by, the symmetric variation \((u_{j}+c)- (u_{j}-c)=2c\) for a sample \(\xi\). We denote \(\one_{j}\) the vector of the same dimension as \(u\) with a 1 in the \(j\)-th component and 0 elsewhere. The FD stochastic estimate of \(P'_{u_{j}}\) is
\begin{equation}
\label{eq-findif}
\gradif{u_{j}}(u,\xi)=\frac{\II\big(\alpha-\theta(u+c\one_{j},\xi)\big)-\II\big(\alpha-\theta(u-c \one_{j},\xi)\big)}{2c}\,.
\end{equation}
It is wise to use the same sample \(\xi\) for the evaluation at \(u+c\) and \(u-c\) in order to reduce variance. A symmetric difference around \(u\) is also recommended. We notice that \cite{LecYin-1997} includes FD under a.s.\ continuity assumptions, which is not our case here, because the indicator functions are discontinuous.

The following theorem provides the analysis of bias and variance of this estimate w.r.t.~the parameter~\(c\).

\begin{theorem}
\label{theo-FD}
If \(P\) (see \eqref{eq-psi}) is three times continuously differentiable with bounded derivatives, the expression~\eqref{eq-findif} provides a biased estimate of \(P'_{u_{j}}\) with a bias in \(\mathrm{O}(c^{2})\).
 If
\begin{description}
\item[(H1)] \(\theta(\cdot,\xi)\) is differentiable with derivatives bounded uniformly in \(\xi\);
\item[(H2)] the probability measure of \(\xi\) has a density;
  \item[(H3)] \(\theta(u,\cdot)\) is twice differentiable and, for all \(u\), and for every solution \(\hat{\xi}\) of \(\theta(u,\xi)=\alpha\), we have that \(\theta'_{\xi}(u,\hat{\xi})\neq 0\);
  \end{description}
  then the variance of estimate~\eqref{eq-findif} is in \(\mathrm{O}(c^{-1})\).

If \emph{\textbf{(H1)}} and \emph{\textbf{(H2)}} still hold true but \emph{\textbf{(H3)}} is replaced by
\begin{description}
  \item[(H4)] \(\theta(u,\cdot)\) is three times  differentiable  and, whenever \(\theta(u,\hat{\xi})=\alpha\) for some \(\hat{\xi}\), and \(\theta'_{\xi}(u,\hat{\xi})=0\), we have that \(\theta''_{\xi^{2}}(u,\hat{\xi})\neq 0\);
  \end{description}
then the variance of estimate~\eqref{eq-findif} is in \(\mathrm{O}(c^{-3/2})\).

Finally, under no particular assumptions on $g$, the best bound for the variance is in \(\mathrm{O}(c^{-2})\). 
\end{theorem}
\begin{proof}
With the smoothness assumption on \(P\), one has that
\begin{align*}
\EE\gradif{u_{j}}(u,\cdot)-P'_{u_{j}}(u)&=\frac{P(u+c\one_{j})-P(u-c \one_{j})-2c\,P'_{u_{j}}(u)}{2c}\\
&=    \frac{c^2}{6} P'''_{u_{j}^{3}}(u) + \mathrm{O}(c^3)\,,
\end{align*}
which proves the claim on the bias.

To evaluate the variance of \eqref{eq-findif}, we study its second order moment which differs from the variance by \(\big(P'_{u_{j}}(u)\big)^{2}\) up to terms in \(\mathrm{O}(c^{2})\) as we have just seen. Consider
\begin{align*}
\EE\Big(\gradif{u_{j}}(u,\xi)\Big)^2&=\EE\Bigg(\frac{\II\big(\alpha-\theta(u+c\one_{j},\xi)\big)-\II\big(\alpha-\theta(u-c \one_{j},\xi)\big)}{2c}\Bigg)^{2}\\
&=\frac{1}{4c^2}\Big(\PP\big(\{\theta(u+c\one_{j},\xi)\leq \alpha\}\cap\{\theta(u-c\one_{j},\xi)>\alpha\}\big)\\&\qquad\qquad+\PP\big(\{\theta(u+c\one_{j},\xi)> \alpha\}\cap\{\theta(u-c\one_{j},\xi)\leq \alpha\}\big)\Big)\,,
\end{align*}
those two events being of course disjoint. 

Using the mean value theorem (or Taylor representation) for the function $\theta(\cdot,\xi)\in C^1$, we have that \(\theta(u+ c\one_{j},\xi) = \theta(u,\xi) + c \theta'_{u_{j}}\big(v^+(\xi),\xi\big)\), and similarly \(\theta(u- c\one_{j},\xi) = \theta(u,\xi) - c \theta'_{u_{j}}\big(v^-(\xi),\xi\big)\). Therefore,
\begin{multline}
\EE\Big(\gradif{u_{j}}(u,\xi)\Big)^2=\frac{1}{4c^2}\Bigg(\PP\Big(\{\alpha+c \theta'_{u_{j}}\big(v^-(\xi),\xi\big) <\theta(u,\xi)\leq \alpha-c \theta'_{u_{j}}\big(v^+(\xi),\xi\big)\}\Big)\\+\PP\Big(\{\alpha-c \theta'_{u_{j}}\big(v^+(\xi),\xi\big) <\theta(u,\xi)\leq \alpha+c \theta'_{u_{j}}\big(v^-(\xi),\xi\big)\}\Big)\Bigg)\,,\label{eq-eval}
\end{multline}
Thanks to \textbf{(H1)}, we can bound each of the above two probabilities by 
\begin{equation}
\label{eq-pbound}
\PP\big(\{\theta(u,\xi) \in (\alpha - Kc, \alpha+ Kc]\}\big)\,,
\end{equation}
where $K$ is the uniform bound on $\theta'_{u_{j}}$.

Our goal is now to evaluate the behavior of this probability when \(c\) is approaching 0. Let \(m\) be the dimension of \(\xi\); \(g\) is supposed to be \(\RR\)-valued. Consider any solution~\(\hat{\xi}\) of\begin{equation}
\label{eq-sol}
\theta(u,\xi)=\alpha\,.
\end{equation}
The case when no such solution exists for some \(u\) will be discussed later on at Remark~\ref{rem-strange}.
If \(\theta'_{\xi}(u,\hat{\xi})\neq 0\) as assumed in \textbf{(H3)}, then the manifold of solutions of \eqref{eq-sol} is locally of dimension less than or equal to \(m-1\). The set of \(\xi\)'s involved in the event in~\eqref{eq-pbound} is locally a set with a ``backbone'' given by this manifold around \(\hat{\xi}\), and a ``thickness'' which is proved to be  of order~\(\mathrm{O}(c)\). Indeed, with a Taylor expansion of \(\theta(u,\cdot)\) around \(\hat{\xi}\), we get
\begin{displaymath}
\theta(u,\hat{\xi}+y)=\alpha+\scal{\theta'_{\xi}(u,\hat{\xi})}{y}+\mathrm{O}(\|y\|^{2})\,.
\end{displaymath}
In this expression, we need only consider \(y\)'s which are (asymptotically as \(c\) goes to 0) parallel to the gradient \(\theta'_{\xi}(u,\hat{\xi})\) (that is, the component in the kernel of the linear form defined by this gradient is useless). It should now be obvious that to match variations of \(g\) around \(\alpha\) which are of order \(c\), we need only consider \(y\)'s which are also of order \(c\) in norm. If this holds true for any \(\hat{\xi}\) in the manifold of solutions of \eqref{eq-sol}, then the probability~\eqref{eq-pbound} is of order \(\mathrm{O}(c)\) and  the second order moment~\eqref{eq-eval} of our estimate (and consequently the variance too) is bounded by an \(\mathrm{O}(c^{-1})\).

If \textbf{(H3)} does not hold but \textbf{(H4)} does, then the same reasoning can be repeated (for a Taylor expansion of the next order) with \(y\)'s which are now orthogonal to the kernel of the Hessian~\(\theta''_{\xi^{2}}(u,\hat{\xi})\) (this component is non zero thanks to \textbf{(H4)}) and it should be clear that to compensate for variations of \(g\) of order \(c\), we now need \(y\)'s which are of order \(\mathrm{O}(c^{1/2})\) in norm. This also gives the order of the probability~\eqref{eq-pbound} and then, the bound on the variance is in \(\mathrm{O}(c^{-3/2})\).

We could continue like that by removing assumption~\textbf{(H4)} but introducing an assumption~\textbf{(H5)}, and so on and so forth. Ultimately, with no particular assumptions, \eqref{eq-pbound}
is of order \(\mathrm{O}(1)\) and the variance is of order \(\mathrm{O}(c^{-2})\).\end{proof}

\begin{remark}\label{rem-strange}
Suppose that for some \(u\), there exists no solutions to \eqref{eq-sol}. Then, since \(g\) is assumed to be at least continuous in \(\xi\), this means that for all \(\xi\), \(\theta(u,\xi)\) is always either strictly less or strictly greater than \(\alpha\), in which cases \(P(u)\) (see \eqref{eq-psi}) assumes either the value 1 or 0 (which are extreme values for \(P\)). 

If \(\theta(u,\cdot)\) can be bounded away from \(\alpha\), then the probability~\eqref{eq-pbound} will be 0 for \(c\) small enough. This is the good case for the variance of the estimate. But \(\theta(u,\cdot)\) may also approach \(\alpha\) asymptotically, and, with heavy tails for the density~\(q\) of \(\xi\), it is not possible to give a better bound for \eqref{eq-pbound} than \(\mathrm{O}(1)\). Here is an example. Let \(\theta(u,\xi)=u-e^{-\xi}\) (\(u\) and \(\xi\) are both scalar). Consider the probability \(\PP(\theta(0,\xi))\in [-c,c]\) for \(c\) small, that is, \(\PP(\xi\geq -\ln c)\). Assume the density \(q(\xi)\) has a positive support and that it is equal to \(a (1+\xi)^{-(1+a)}\, \II(\xi)\) with \(a\) an arbitrary small positive number. Then \(\PP(\xi\geq -\ln c)= (1-\ln c)^{-a}\). For \(c\) positive and below \(e\), \((1-\ln c)^{-1}\geq c\), hence this probability is larger than \(c^{a}\). Since \(a\) is positive and arbitrarily small, we cannot clearly make this case enter the case of better bounds obtained with assumptions \textbf{(H3)} or \textbf{(H4)}. 
\end{remark}
\subsubsection{Application to the Example and Comparison with the AC Method}
We have used \eqref{eq-findif} (for the two components of the gradient, that in \(u\) and that in \(v\)) to our example and evaluated, once again with the help of \emph{Mathematica}, the mean and variance of those estimates at the optimal solution~\eqref{eq-sharp}. The results are as follows:
\begin{align*}
\EE\gradif{u}(u^{\sharp},v^{\sharp},\cdot)& =0.62-0.23 c^2+0.06 c^4+\mathrm{O}(c^7)\,,\\
\var\gradif{u}(u^{\sharp},v^{\sharp},\cdot)  & = \frac{0.31}{c} - 
    0.39- 0.12 c +
      \mathrm{O}(c^2)\,,\\
\EE\gradif{v}(u^{\sharp},v^{\sharp},\cdot) & = 1.18 - 1.49 c^2-42.25 c^4 -
     199.41 c^6+\mathrm{O}(c^7) \,,\\
\var\gradif{v}(u^{\sharp},v^{\sharp},\cdot)  & =\frac{0.59}{c} - 
    0.39- 0.74 c +
      \mathrm{O}(c^2)\,.
\end{align*}
Following the same procedure as for the AC estimate, the MQE for the gradient vector estimate based on \(N\) independent samples is obtained by \(\var(c)/N +( \EE(c))^2-( \EE(0))^2\); in this expression,  the dominant terms in \(1/Nc\) and in \(c^4\) only are retained to tune \(c\) as a function of~\(N\). This yields \(c=0.63 N^{-1/5}\) and an optimal MQE equal to:
\begin{equation}
 \frac{1.79}{N^{4/5}}-\frac{1.78}{N}+\mathrm{O}(N^{-6/5})\;.\label{eq-mqedif}
\end{equation}
Compared with \eqref{eq-mqe} which was obtained with the AC estimate, this is asymptotically slightly better. However, a more careful inspection with complete expressions of the MQEs shows that this conclusion becomes true only for \(N\) above about 11000. Hence one may say that the AC and the FD  methods yield approximately the same performances.

\section{Convergence Analysis}\label{section5}
\subsection{Stochastic Algorithms}
Consider algorithm~\eqref{eq-ahs}. With \(\Theta(u)=\EE \big(\theta(u,\xi)\big)\) and \(J(u)=\EE j(u,\xi)\), an equilibrium point \((u^{\sharp},\lambda^{\sharp})\) of this algorithm solves the system of Kuhn-Tucker optimality conditions of problem~\eqref{eq-pb}: for all positive \(\varepsilon\) and \(\rho\),
\begin{subequations}\label{eq-KT}
\begin{align}
& u^{\sharp}=\Pi_{U^{\mathrm{ad}}} \Big(
u^{\sharp}-\varepsilon\big(\nabla_{u}J(u^{\sharp})+\nabla_{u}\Theta(u^{\sharp})\,\lambda^{\sharp}\big) \Big)\,,\label{eq-KT-a}\\
& \lambda^{\sharp}=\Pi_{+}\Big(
\lambda^{\sharp}+\rho\big(\Theta(u^{\sharp})-\alpha\big) \Big)\,.\label{eq-KT-b}
\end{align}
\end{subequations}
We will write algorithm~\eqref{eq-ahs} (with \(\rho^{k}\) proportional to \(\varepsilon^{k}\)) more compactly: we set \(x=(u,\lambda)\) and write
\begin{equation}
\label{eq-sa}
x^{k+1} = \Pi(x^k - \varepsilon^k \, \psi^k)\,,
\end{equation}
where \(\Pi\) stands for the projection operation on \(\Uad\times\RR^{d}_{+} \) and \(\psi^k\) is driven by an underlying process of i.i.d.~drawings \(\xi^{k+1}\), independent of \(\{x^{i}\}_{i\leq k}\). Let $\Field^{k}$ be the filtration generated by \(\{x^{k},\{\xi^{i}\}_{i\leq k}\}\) so that $\psi^k$ and $x^{k+1}$ are $\Field^{k+1}$ measurable.

With the stochastic estimates produced by the AC and FD techniques considered so far in this paper, we obtained \emph{biased} estimates of \(\nabla_{u}\Theta\) (and the bias sometimes also affects the estimate of \(\Theta\) itself), with a bias going to 0 	as \(k\to +\infty\). We will denote \(\Psi(x^{k})\) the correct value of the vector field at \(x^{k}\), namely
\begin{subequations}\label{eq-moyen}
\begin{align}
&\nabla J(u)+\nabla \Theta(u)  \lambda\,,  \\
&\alpha-\Theta(u)\,,  
\end{align}
\end{subequations}
that with which an equilibrium point satisfies (see \eqref{eq-KT}):
\begin{equation}
\label{eq-equil}
x^{\sharp}=\Pi\big(x^{\sharp}-\varepsilon\Psi(x^{\sharp})\big)
\end{equation}
for all positive \(\varepsilon\). 

Define the martingale difference $\Delta M^{k}$, the bias $B^{k}$ and the variance $V^{k}$ of $\{\psi^k\}$ by:
\begin{subequations}\label{eq-quant}
\begin{align}
&\Delta M^{k} = \psi^k - \EE(\psi^k \mid \Field^{k})\,,\\
&B^{k} = \EE(\psi^k\mid\Field^{k}) - \Psi(x^k)\,,\\
&V^{k} = \EE\|\psi^k - \EE(\psi^k\mid \Field^{k})\|^2\,.
\end{align}
\end{subequations}
We will use references \cite{KusYin:2003} and \cite{LecYin-1997} in which convergence results and convergence rates of algorithm~\eqref{eq-sa} are provided. Essentially, if the nonlinear projection operation at the r.h.s.~of \eqref{eq-sa} is missing, under conditions on the quantities~\eqref{eq-quant} in connection with the step size \(\varepsilon^{k}\) that we will recall below, the trajectory produced by \eqref{eq-sa} behaves a.s.\ as that of the deterministic ODE:
\begin{subequations}\label{eq-ODE}
\begin{equation}
\label{eq-ODEa}
\dot{x}= -\Psi(x)\,.
\end{equation}
In the presence of the projection onto a closed convex set, the differential equation is more complex to write since it involves another process \(z\) in which \(z\) takes values in the orthogonal cone~\(C(x)\) to the convex set at the current point~\(x\) (hence this process effectively appears only at the border of the convex set). The ODE now reads
\begin{equation}
\label{eq-ODEb}
\dot{x}= -\Psi(x)-z\,,\quad z\in C(x)\,.\end{equation}
\end{subequations}
The role of \(z\) is to maintain \(x\) in the convex set, as it is the case for \(x^{k}\) produced by \eqref{eq-sa}. It is defined as the ``minimum force'' which achieves this goal.

\subsection{Convergence}
In this subsection, we recall the conditions which ensure that the stochastic Arrow-Hurwicz algorithm will behave as its ODE~\eqref{eq-ODE} and we refer to the previous subsection to deduce that primal iterates~\(u^{k}\) will converge, at least locally, towards the solution~\(u^{\sharp}\) (assumed unique) of the constrained optimization problem. We then apply those results to the case of biased gradient estimates provided by AC and FD methods to derive a policy on how to tune the parameters \(r\) (see \eqref{eq-bias}) and \(c\) (see \eqref{eq-findif}) as functions of the iteration index~\(k\) in order to satisfy the convergence conditions.

\begin{lemma}\label{lem-conv}
Consider the iteration~\eqref{eq-sa} and assume that 
\begin{subequations}\label{eq-assump}
\begin{align}
&\sum_k \varepsilon^k = +\infty\,,\label{eq-assump-a} \\ 
&\sum_k \varepsilon^k \| B^k\| <\infty \quad\text{a.s.},\label{eq-assump-b}\\
&\sum_k (\varepsilon^k)^2\, V_k <\infty\,.\label{eq-assump-c}\end{align} 
\end{subequations}
Then, a.s.,  $x^k$ has the same asymptotic behavior as the solution of \eqref{eq-ODE}. 
\end{lemma}
This result follows from \cite[Chap.~5]{KusYin:2003}. 

\begin{proposition}
Consider the case when the estimate~\eqref{eq-bias} (and possibly \eqref{eq-ther} too) is (are) used in the stochastic algorithm~\eqref{eq-ahs} (with \(\rho^k\) proportional to \(\varepsilon^k\)) with the following choices of the stepsize~\(\varepsilon^{k}\) and of the ``mollifier'' parameter~\(r^{k}\):
\begin{equation}
\label{eq-rn}
\varepsilon^{k} = k^{-\gamma}, \quad r^{k} = k^{-\beta/2}\,,
\end{equation}  
for \(\beta\) and \(\gamma\) positive.
Then the conditions of Lemma~\ref{lem-conv} are satisfied if\begin{equation}
\label{eq-conds}
\gamma\leq1, \quad\beta + \gamma > 1,\quad 2\gamma - \beta/2>1\,.
\end{equation}
\end{proposition}
\begin{proof}
The first condition~\eqref{eq-conds} is required by \eqref{eq-assump-a}. Theorem~\ref{th-1} states that the bias~\(B^{k}\) of AC estimates is in \(\mathrm{O}\big((r^k)^{2}\big)=\mathrm{O}\big(k^{-\beta}\big)\), hence \(\varepsilon^{k}\|B^{k}\| = O\big(k^{-( \beta+\gamma)}\big)\); therefore \eqref{eq-assump-b} is satisfied under the second condition~\eqref{eq-conds}. As for the variance~\(V^{k}\), it is in \(\mathrm{O}\big((r^{k})^{-1}\big)=\mathrm{O}\big(k^{\beta/2}\big)\) which yields $(\varepsilon^k)^2\, V^{k} = \mathrm{O}\big(k^{\beta/2-2 \gamma }\big)$; thus \eqref{eq-assump-c} is satisfied under the third condition~\eqref{eq-conds}. 
\end{proof}
\begin{proposition}
Consider the case when the estimate~\eqref{eq-findif} is  used in~\eqref{eq-ahs-a} (with \(\rho^k\) in \eqref{eq-ahs-b} proportional to \(\varepsilon^k\)) with the following choices of the stepsize~\(\varepsilon^{k}\) and of the FD parameter~\(c^{k}\):
\begin{equation}
\label{eq-cn}
\varepsilon^{k} = k^{-\gamma}, \quad c^{k} = k^{-\beta/2}\,,
\end{equation} 
for \(\beta\) and \(\gamma\) positive. 
Then the conditions of Lemma~\ref{lem-conv} are satisfied if, in addition of assumptions \emph{\textbf{(H1)}} and \emph{\textbf{(H2)}} of Theorem~\ref{theo-FD}, one has that
\begin{equation}
\label{eq-condd}
\gamma\leq1,\quad\beta + \gamma > 1,\quad
\begin{cases}
  2\gamma - \beta/2>1    & \text{if \emph{\textbf{(H3)}} is satisfied in Theorem~\ref{theo-FD}}, \\
  2\gamma - 3\beta/4>1    & \text{if \emph{\textbf{(H4)}} is satisfied in Theorem~\ref{theo-FD}}, \\
  2\gamma - \beta>1     & \text{otherwise}.
\end{cases}
\end{equation}
\end{proposition}
The proof follows the same pattern as previously using the evaluations of Theorem~\ref{theo-FD}, the only changes concerning \(V^{k}\).

\subsection{Convergence Rate}\label{sec-rate}
Let \(\beta\) and \(\delta\) be the integers such that:
\begin{equation}
\label{eq-betadelta}
B^{k} = \mathrm{O}(k^{-\beta}), \quad V^{k} = \mathrm{O}(k^{-\delta})\,. 
\end{equation}
Reference \cite{LecYin-1997} provides a comprehensive analysis of the convergence rates of algorithms of type \eqref{eq-sa} under the following assumption: close to its unique equilibrium point \(x^{\sharp}\) (supposed to lie in the interior of the convex set onto which \(\Pi\) is the projection), function \(\Psi\) admits the following representation:
\begin{equation}
\label{eq-A}
 \Psi(x) = A(x-x^{\sharp}) + \mathrm{O}(\|x-x^{\sharp}\|^2)\,,
\end{equation} 
where \(A\) is a matrix with  eigenvalues~\(\mu\) satisfying
\begin{equation}
\label{eq-lambdamin}
\overline{\mu}=\min\big(\mathrm{Re}(\mu)\big) > \begin{cases}0
      & \text{if \(\gamma<1\)}, \\
 \max \big(\beta, (1+\delta)/2\big) & \text{if \(\gamma=1\)}.
\end{cases} 
\end{equation}
Then, a direct application of Theorem 3.1 in \cite{LecYin-1997} gives the asymptotic mean square error (MSE) as a function of the algorithm parameters $\gamma, \beta, \delta$ and it states that:
\begin{equation}
\label{eq-rates}
\EE(x^k-x^{\sharp})^2 = \mathrm{O}(k^{-\kappa}), \quad \kappa = \min(2\beta, \gamma + \delta).
\end{equation}

Some comments are in order here regarding the application of this result to our situation. First, the authors of \cite{LecYin-1997} state than when \(x^{\sharp}\) lies on the boundary of the feasible convex set, other techniques (e.g.~large deviations) are required to establish convergence rates. In our case, we expect that the probability constraint is active at the optimum, hence the optimal dual variable should be strictly positive. In the example of \S\ref{sec-example}, we also have positivity constraints on primal variables and that on \(u\) (the first primal component) is active at the optimum (see \eqref{eq-sharp}). But it is felt that the projection is rather helpful in accelerating convergence for this component (see numerical results in the next section). We may consider that, asymptotically, \(u^{k}\) is ``frozen'' at~0 and does not participate to the dynamics of the algorithm ultimately.

Second, condition~\eqref{eq-lambdamin} may not be satisfied. We will come back on this point in the next subsection.  Nevertheless, we used the results of \cite{LecYin-1997} as guidelines for the choice  of parameters \(\beta\) and \(\gamma\) to drive the primal solution to its equilibrium in the most efficient way.

That said, in order to achieve the fastest convergence rate, one should seek to maximize $\kappa$ in \eqref{eq-rates} over the feasible set defined by \eqref{eq-conds} or \eqref{eq-condd} and the expression of \(\delta\) as a function of \(\beta\).
For the case of AC estimates, \(\delta=-\beta/2\), the minimum of \(2\beta\) and \(\gamma-\beta/2\) is obtained when those two functions are equal, which yields \(\beta=2\gamma/5\) and a value of \(4\gamma/5\); because of the first condition~\eqref{eq-conds}, the maximal possible value is obtained with \(\gamma=1\), which yields \(\beta=2/5\) and \(\kappa=4/5\), and we check that this pair \((\beta,\gamma)\) satisfies all conditions in \eqref{eq-conds}. Observe that our heuristic reasoning at the end 
of \S\ref{sec-gen} and \S\ref{sec-sec} in order to tune the parameter \(r\) when \(N\) i.i.d.~samples are available (here \(N\) is the iteration index~\(k\)) yields the same results (see \eqref{eq-mqe} in particular).

For the case of FD estimates, under assumption \textbf{(H3)} of Theorem~\ref{theo-FD}, the calculations and conclusions are the same. Under assumption~\textbf{(H4)}, \(\delta=-3\beta/4\) and the optimal values are \(\beta=4/11\), \(\gamma=1\), \(\kappa=8/11\) which is of course worse than the previous case. Finally, in the worst case for FD, we get \(\beta=1/3\), \(\gamma=1\), \(\kappa=2/3\).

The following result is a direct application of \cite[Th.~4.1 and 4.2]{LecYin-1997}. This CLT gives additional information on the asymptotic behavior of the iterates of~\eqref{eq-sa}.
\begin{theorem}
\label{thm-FCLT}
Consider algorithm~\eqref{eq-sa} with assumptions~\eqref{eq-A}, \eqref{eq-betadelta} and $\varepsilon^k=1/k$ (that is, \(\gamma=1\) in \eqref{eq-rn} or \eqref{eq-cn}). Let 
\[
X^{k} = k^{\kappa/2} (x^{k} - x^{\sharp})\,,\]
with \(\kappa\) as in \eqref{eq-rates}.
If $2\beta \ge 1+\delta$, then  as $k\to\infty$, 
\(
X^{k} - k^{\kappa/2-\beta} H_b \bar{B}
\)
converges in distribution towards a normal distribution of mean 0 and covariance~\(\Sigma\)
where:
\begin{align*}
& \bar{B} = \lim_{k\to\infty} k^\beta B^{k}\,,\\
& H_b = A - \beta I\,,\\
& H = A -\big((1+\delta)/2\big)I\,, \\
& R = \lim_{k\to\infty} k^\delta\,  \EE(\Delta M^{k} (\Delta M^{k})^{\top} \mid \Field^{k})\,,\\
& \Sigma H + H^{\top} \Sigma = R\,,
\end{align*}
where \(\null^{\top}\) denotes transposition.
\end{theorem}

\begin{remark}
From the definition of $H_b$ above and the appearance of $A -\big((1+\delta)/2\big)I$ in the definition of $\Sigma$, it is apparent that the strong stability condition~\eqref{eq-lambdamin} (case \(\gamma=1\)) ensures that both these matrices are positive definite, so that \(\Sigma\) is well defined. Indeed, with our choices, \(H_{b}\) and \(H\) are equal.
\end{remark}

\subsection{The Case of Arrow-Hurwicz Algorithms}\label{sec-arhu}
We know discuss the properties of matrix~\(A\) in the situation of Arrow-Hurwicz algorithms. This matrix has been introduced in \eqref{eq-A} in general,  and the operator \(\Psi\) is defined by~\eqref{eq-moyen} in our case. Thus, \(A\) is the linearized version of that \(\Psi\) at the equilibrium point~\(x^{\sharp}\), that is,
\begin{equation}
\label{eq-arr}
A=\begin{pmatrix}
\frac{\partial^{2}L(u^{\sharp},\lambda^{\sharp})}{\partial u^{2}}&\frac{\partial^{2}L(u^{\sharp},\lambda^{\sharp})}{\partial u\, \partial \lambda}\\
-\frac{\partial^{2}L(u^{\sharp},\lambda^{\sharp})}{\partial \lambda\,\partial u}&-\frac{\partial^{2}L(u^{\sharp},\lambda^{\sharp})}{\partial \lambda^{2}}
\end{pmatrix}=\begin{pmatrix}
J''(u^{\sharp})+\big(\lambda^{\sharp}\big)^{\top}\Theta''(u^{\sharp})&\big(\Theta'(u^{\sharp})\big)^{\top}\\
-\Theta'(u^{\sharp})&0
\end{pmatrix}
\end{equation}
However, among the constraints \(\Theta\), only those saturated (that is, satisfied with equality) at the equilibrium point should be taken into account together with their corresponding multipliers (that is, the non saturated constraints are virtually absent asymptotically).

Under the assumptions that the gradients of saturated constraints are linearly independent (or, otherwise stated, the operator in the upper right-hand corner of the matrix is injective), and that the Hessian of the Lagrangian (that is, the operator in the upper left-hand corner) is positive definite, it can easily be proved that the real part of the eigenvalues of~\(A\) are positive (see \cite[proof of Proposition~4.4.2]{berts}). This is condition~\eqref{eq-lambdamin} in the case \(\gamma<1\). When \(\gamma=1\), condition~\eqref{eq-lambdamin} is stronger and will be discussed shortly in the case of our example. Observe that if we assume that the only saturated dualized constraint is the probability constraint (which is the case in our example), then we should assume that the gradient of this probability function at the equilibrium is not zero.
 
Going back to example of \S\ref{sec-example}, matrix~\(A\) (restricted to the variables \((u,v,\lambda_{2})\)) is equal to
\begin{displaymath}
\begin{pmatrix}
0.944 &1.002 &-0.621\\
1.002&1.211&-1.181 \\
0.621 &1.181 & 0
\end{pmatrix}
\end{displaymath}
with eigenvalues \(0.974\pm 0.753~i\) and \(0.207\). As predicted, the real parts are positive but the smallest one is equal to \(0.207\) which is \emph{not} greater than \(2/5\). Thus, condition~\eqref{eq-lambdamin} (case \(\gamma=1\)) is not satisfied (with \(\beta=(1+\delta)/2= 2/5)\). However, in the same way as we ignored multipliers corresponding to non saturated constraints because they are stuck to 0 asymptotically, we may consider that the part~\(u\) of primal variables is ``out of the game'' ultimately because \(u\) is stuck to~\(0\) (the constraint \(u\geq 0\) is saturated) at the end of the transient part of the algorithm (remember that the ODE~\eqref{eq-ODEa} is to be replaced by the more complex dynamics~\eqref{eq-ODEb} when following boundaries of the admissible domain). Therefore, we consider a reduced matrix \(A\) by keeping only the \(2\times 2\) lower right-hand block (corresponding to the pair \((v,\lambda_{2})\)). The eigenvalues of this reduced matrix are \(0.605\pm  1.014~i\) and now condition~\eqref{eq-lambdamin} is satisfied even for the case \(\gamma=1\).
\section{Numerical Results}\label{section6}

Algorithm~\eqref{eq-ahs} has been used to solve the example of \S\ref{sec-example} with the AC and FD estimates. 
\begin{subequations}\label{eq-algoex}
\begin{align}
& u^{k+1}=\Pi_{U^{\mathrm{ad}}} \Big(
u^{k}-\varepsilon^{k}\big(\nabla_{u}j(u^{k},\xi^{k+1})-\estgrad{u}(u^{k},\xi^{k+1})\,\lambda^{k}\big) \Big)\,,\\
& \lambda^{k+1}=\Pi_{+}\Big(
\lambda^{k}+\rho^{k}\big(\ppp-\widehat{P}(u^{k+1},\xi^{k+1})\big) \Big)\,.
\end{align}
\end{subequations}
More precisely, for the AC method, \(\estgrad{u}(u,\xi)\) should be interpreted as the gradient estimate~\eqref{eq-bias}, applied to the example (see \eqref{eq-estuv}); \(\widehat{P}(u,\xi)\) is either \(p(u,\xi)\) as in \eqref{eq-puxi} or the biased estimate given by \eqref{eq-ther} (see Remark~\ref{rem-variante}). We tested both versions numerically and there was no significant difference. The estimate~\eqref{eq-ther} was retained for the rest of experiments. Of course, parameter \(r^k\) is adjusted according to the rule \(r^k = a k^{-1/5}\) where \(a\) is a positive constant to be tuned.

For the FD method, \(\estgrad{u}(u,\xi)\) is given by \eqref{eq-findif},  applied to the example. Again, parameter \(c^{k}\) is adjusted as \(c^k=b k^{-1/5}\) where \(b\) is a positive constant to be tuned. For \(\widehat{P}(u,\xi)\), we used~\eqref{eq-puxi}.

Numerical experiments are performed according to the following protocol:
 \begin{itemize}
\item all runs of the algorithms start from the same initial conditions:
\begin{displaymath}
u^{0}=0.2,\quad v^{0}=0.8,\quad \lambda_{1}^{0}=0.5, \quad\lambda_{2}^{0}=0.3\,.
\end{displaymath}
Recall that the solution is given by \eqref{eq-sharp} and all results will be expressed in terms of differences with those optimal values (hence the equilibrium point for all variables is at 0).
\item For AC and FD, 100 runs of the algorithms are performed using the same 100 sequences of pseudo-random numbers to generate Monte Carlo samples of \(\xi\) according to the distribution of this variable.
\item 5000 thousands iterations are performed for each run.
\item For AC and FD, averages of the differences \(x^{k}-x^{\sharp}\) are computed over the 100~runs together with their standard deviations. What will be shown on the plots are the trajectories of the ``average \(\pm\) standard deviation'' of those quantities as functions of the iteration index~\(k\).
\item The parameters \(a,b,d,e,f,g\) appearing in the following rules:
\begin{displaymath}
r^{k}=\frac{a}{k^{1/5}},\quad c^{k}=\frac{b}{k^{1/5}},\quad \varepsilon^{k}=\frac{d}{e+k},\quad \rho^{k}=\frac{f}{g+k},
\end{displaymath}
are tuned by some trials to try to obtain the ``best'' results for both methods.
\end{itemize}
\begin{figure}[hbtp]
   \centering
   \includegraphics[scale=0.45]{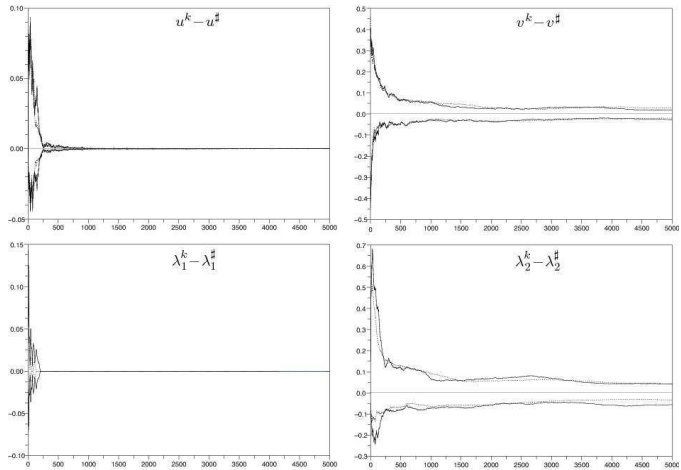}
      \caption{Average \(\pm\) standard deviation for AC (solid line) and FD (dotted line) algorithms}
   \label{fig-plot}
\end{figure}
Figure~\ref{fig-plot} shows the plots for the four variables and for the AC (continuous line) and DF (dotted line) methods. Again what is displayed is the ``average \(\pm\) standard deviation'' over 100~runs. Results obtained on this example are very close (with maybe a slight advantage to FD in the earliest iterations) with both methods. This confirm the estimation of variance and bias made with \emph{Mathematica} around the optimum for the estimates obtained with the two methods.

\section{Conclusions}
This paper discussed the problem of stochastic optimization under probability constraints and in particular methods for solving them numerically. Although there exist other ways of taking care of risk considerations in decision problems under uncertainty, we discussed the fact (\S\ref{sec-risk}) that probability constraints are sometimes the most straightforward way of expressing and quantifying risk in some circumstances. 

Unfortunately, as shown by the discussion and examples in \S\ref{section3}, probability constraints may be the source of several pathologies, and the loss of convexity is the most frequent one. Nevertheless, one must address the problem of numerical resolution with approaches which may fail in the worst cases but which may also succeed to solve nontrivial problems. Our strategy is based on duality and stochastic gradient algorithms. Duality, and the use of stochastic Arrow-Hurwicz algorithms, require the existence of a saddle point of the Lagrangian, which is not granted for the reasons advocated above. The use of augmented Lagrangians would certainly increase the chance of existence of saddle points but, in combination with stochastic algorithms, it raises new difficulties (namely, the operator of mathematical expectation would appear inside a nonlinear function). This new topic will be addressed in a forthcoming paper. 

Apart from this problem of saddle point existence, the search of this saddle point by stochastic gradient algorithms is made possible by expressing the probability constraint as an expectation involving a discontinuous function. In this paper, we proposed two ways to overcome this difficulty, and we studied the convergence and convergence rate of the resulting algorithms. The two methods provide \emph{biased} stochastic estimates of the constraint gradient. Although their implementation on a simple example showed a similar behavior, the theoretical results reveal that in more general situations, the ``mollifier'' (or ``Approximation by Convolution'' --- AC) method should be of more general use and robustness then the ``Finite Difference'' (FD) method. We defer to a forthcoming paper to propose other estimation techniques providing \emph{unbiased} estimates and based on techniques of integration by parts.

Still, the surface of this difficult field of numerical resolution of probability constrained stochastic optimization problems has been just scratched here, and several directions remain open for future investigations. For example, we have considered here only events (whose probability is constrained) which are described only by a scalar constraint and the case of events described by multidimensional constraints may raise new questions (although the techniques discussed in the present paper seem ready for an extension to this case).

\end{document}